%% file: paper.tex
\documentclass[onefignum,onetabnum]{siamart190516}

\input{shared}

\ifpdf
\hypersetup{
  pdftitle={Finite elements for Helmholtz equations with a nonlocal boundary condition},
  pdfauthor={R. C. Kirby, A. Kl\"ockner, and B. Sepanski}
}
\fi

\usepackage{graphicx}
\usepackage{float}

\begin{document}

\maketitle

% REQUIRED
\begin{abstract}
  Numerical resolution of exterior Helmholtz problems requires some
  approach to domain truncation.  As an alternative to approximate
  nonreflecting boundary conditions and invocation of the
  Dirichlet-to-Neumann map, we introduce a new, nonlocal boundary
  condition.  This condition is exact and requires the evaluation of
  layer potentials involving the free space Green's function.
  However, it seems to work in general unstructured geometry, and Galerkin
  finite element discretization leads to convergence under the usual
  mesh constraints imposed by G{\aa}rding-type inequalities.  The
  nonlocal boundary conditions are readily approximated by fast
  multipole methods, and the resulting linear system can be preconditioned by
  the purely local operator involving transmission boundary conditions. 
\end{abstract}

% REQUIRED
\begin{keywords}
  finite element method, Helmholtz, boundary conditions, layer potentials
\end{keywords}

% REQUIRED
\begin{AMS}
  65N30, 65N80, 65F08
\end{AMS}

\section{Introduction}

The exterior Helmholtz problem plays an essential role in scattering problems and also serves as a starting point to consider exterior problems in electromagnetics and problems in other unbounded domains such as waveguides.  The literature contains several techniques to address the challenge that unbounded domains pose for numerical methods.   Essentially, these techniques include some combination of truncating the domain to a bounded one, posing boundary conditions that enforce (or approximate) the Sommerfeld condition on the newly-introduced boundary, and/or modifying the PDE near the computational boundary to absorb any reflected waves.

An early paper on finite elements for the exterior problem is~\cite{goldstein1982exterior}, where the domain is truncated at radius $R$ and an approximate radiation condition is posed at $R$.  Although the error estimates contain a factor of $R^{-2}$, it is also possible to carefully increase the mesh spacing near the boundary, somewhat mitigating the cost of a large domain. \emph{Perfectly matched layers}~\cite{berenger1994perfectly} modify the PDE near the boundary of the computational domain, changing the coefficient of the elliptic term to `absorb' outgoing waves.  While such methods allow small effective computational domains, the resulting linear systems do not yield readily to standard iterative techniques like multigrid, although we refer to recent work~\cite{safin2018preconditioned} that poses a domain decomposition strategy to use a direct solver only near the boundary and standard iterative techniques inside.

There is also considerable literature on \emph{nonlocal} boundary conditions for domain truncation.  Following early work~\cite{goldstein1982finite,hou1985approximation,KELLER1989172}, one can use a Dirichlet-to-Neumann (DtN) operator on the artificial boundary to enforce proper far-field behavior.  Givoli~\cite{givoli1991non} provides a survey of similar techniques and local conditions as well, and~\cite{givoli1992spatially, grote1995exact} give techniques for the time rather than frequency domain case.  The DtN is typically given as an infinite series (truncated in computation) obtained by separating variables.  This limits the shape of the domain boundary, although perturbations of such domains and use of high-order methods are possible~\cite{binford2009exact,NICHOLLS2004278,nicholls2006error}.  Careful error analysis for finite element discretizations can include the effect of truncating the infinite series as well as polynomial approximation error~\cite{koyama2007error}.   Lastly, it is worth noting that boundary integral equation methods solve exterior Helmholtz problems with optimal complexity (linear in the number of boundary degrees of freedom), however they are somewhat more difficult to adapt than (volume) PDE discretizing-methods to specific (and potentially nonlinear) near-surface physics.

In this paper, we propose an alternative nonlocal boundary condition based on Green's Theorem~\cite{steinbach2007numerical} that has several important features.  Like the DtN approach, we have an (in principle) exact boundary condition, incurring no error in our domain truncation.  However, because we rely on the free-space Green's function, there is (again, in principle) no restriction on the shape of our computational domain.   The layer potentials appearing in our nonlocal boundary condition can be efficiently computed by appropriate fast algorithms such as variants of the Fast Multipole Method~\cite{carrier_fast_1988}.  So, although Galerkin's method would give matrices with dense sub-blocks, we can quickly compute the matrix-vector product required in a Krylov method.  Because our nonlocal operator involves double integrals over distinct boundaries, we avoid the need to evaluate any singular integrals.  We note that ``two-boundary'' approaches have been explored in the time-domain literature~\cite{givoli1995nonreflecting,hoch2008nonreflecting,ting1986exact}.
Finally, the local part of the operator (a standard finite element matrix) serves as an excellent preconditioner for the system, so that an optimal solver for the local part would give $\mathcal{O}(n \log n)$ solution time in unstructured geometry.  Our method works equally well in two and three space dimensions.  

Another method combining boundary integral and volumetric discretizations is due to Johnson and N\'ed\'elec~\cite{johnson1980coupling, sayas2009validity}.  This technique encloses a compactly-supported volume source in a truncating boundary.  Finite elements are used to compute the solution inside the boundary and a boundary integral method is used on the boundary to handle the exterior.  Our present method bears some similarly, employing the same kind of operators.  However, we require only a single finite element space and do not introduce additional unknowns on the domain boundary.

In the rest of the paper, we pose the model and its finite element discretization in Section~\ref{sec:model}.  We describe a preconditioned Krylov system for this system in Section~\ref{sec:solve}.  Our implementation, which relies on the high-level codes \firedrake{}~\cite{firedrake} and \pytential{}~\cite{pytential}, warrants some discussion, which is given in Section~\ref{sec:imp}.  Finally, we give numerical results in Section~\ref{sec:results}.

\section{Model and discretization}%
\label{sec:model}
Let $\Omega^c \subset \mathbb{R}^{d}$ with $d=2, 3$ be a bounded
domain with boundary $\Gamma$, and
$\Omega = \mathbb{R}^d \backslash \Omega^c$ its exterior.
We consider the classic Helmholtz exterior problem on $\Omega$
\begin{equation}
\label{eq:helmholtz}
-\Delta u - \kappa^2 u = 0,
\end{equation}
where $\kappa^2$ is nonzero wave number.  It may be complex (typically with positive real part), and may take different forms in various application fields such as acoustics or electromagnetics.
We also pose Neumann boundary conditions
\begin{equation}
\label{eq:neumann}
\tfrac{\partial u}{\partial n} = f
\end{equation}
on the interior boundary $\Gamma$.  The Sommerfeld radiation condition
\begin{equation}
    \lim_{r \rightarrow \infty} r^{\tfrac{d-1}{2}} \left(
        \tfrac{\partial u}{\partial r} - i \kappa u
    \right) = 0,
\label{eq:sommerfeld}
\end{equation}
where $r$ is the outward radial direction, must also hold.
For computational purposes, one typically poses the problem only on a
truncated domain $\Omega^\prime$.  Hence, we impose an artificial
boundary $\Sigma$, and let $\Omega^\prime$ denote that subset of
$\Omega$ enclosed between $\Gamma$ and $\Sigma$.  We assume that these boundaries are such that $\Omega^\prime$ is a Lipschitz domain.
An example is shown in Figure~\ref{fig:2ddomain}:
\begin{figure}[!ht]
    \begin{center}
    \begin{tikzpicture}[scale=0.8]
      \draw [thick,fill=black!30] (-3,-3) rectangle (3,3)
        (0, 0) circle (1);
      \node at (0,0) {$\Omega^c$};
      \node at (-2,2) {$\Omega'$};
      \node [anchor=west] at (1,0) {$\Gamma$};
      \node [anchor=west] at (3,0) {$\Sigma$};
    \end{tikzpicture}
    \end{center}
    \caption{A 2D example of a truncated domain}
    \label{fig:2ddomain}
\end{figure}
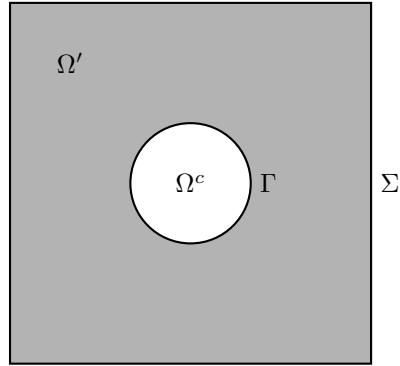

A major challenge for volume-discretizing numerical methods is imposing a suitable boundary condition on $\Sigma$.  For example,
a simple approach is to impose the Robin-type condition
\begin{equation}
\label{eq:transmission}
i \kappa u - \tfrac{\partial u}{\partial n} = 0
\end{equation}
on $\Sigma$ rather than at infinity.  Frequently called ``transmission'' boundary conditions, this changes the boundary value problem, incurring errors that do not vanish under mesh refinement, and can create artificial wave reflections at the boundary.

We propose a new approach to the problem that, for constant-coefficient problems at least, allows highly effective iterative solvers to be combined with effective domain truncation.
Let
\begin{equation*}
  \mathcal K(x):=
    \begin{cases}
        u(x) = \frac{i}{4} H_0^{(1)}(\kappa |x|) & d = 2, \\
        u(x) = \frac{i}{4\pi|x|} e^{i\kappa |x|} & d = 3.
    \end{cases}
\end{equation*}
be the free-space Green's function for the Helmholtz equation, where $H_0^{(1)}(x)$ be the first-kind Hankel function of index 0.
Recall Green's formula in the exterior~\cite[Thm.~2.5]{colton_inverse_1998} for the solution $u(x)$ to~\eqref{eq:helmholtz}:
\begin{equation}
\label{equation:uDL}
u(x) = \int_\Gamma \left( \tfrac{\partial}{\partial n} \mathcal K(x-y ) \right)u(y) - \left(\tfrac{\partial}{\partial n} u(y) \right) \mathcal K( x - y ) dy,
\end{equation}
for $x\in\Omega'$.  It is known that Green's theorem holds in general Lipschitz domains~\cite{rjasanow2007fast}.

  Substituting in the Neumann boundary condition~\eqref{eq:neumann}, we obtain
\begin{equation}
\begin{split}
  u(x) &= \int_\Gamma \left( \tfrac{\partial}{\partial n} \mathcal{K}
  \left(\left|x-y\right| \right) \right)u(y) - f(y) \mathcal{K}\left( \left| x - y \right| \right) dy \\
& \equiv D(u)(x) - S(f)(x),
\end{split}
\end{equation}
for $x\in\Omega'$, where
\begin{equation}
D(u)(x) = \operatorname{PV} \int_\Gamma \left( \tfrac{\partial}{\partial n} \mathcal K( x-y) \right)u(y) dy
\end{equation}
is the double layer potential and
\begin{equation}
S(u)(x) = \int_\Gamma  \mathcal K (x - y ) u(y)dy
\end{equation}
is the single layer potential~\cite{colton_inverse_1998}.

Now, we can pose an exact nonlocal Robin-type boundary condition as follows.  We use the representation~\eqref{equation:uDL} to write (suppressing the argument $x$):
\begin{equation}
i \kappa u - \tfrac{\partial u}{\partial n}
= i \kappa \left( D(u) - S(f) \right) - \tfrac{\partial}{\partial n} \left( D(u) - S(f) \right),
\end{equation}
so that over $\Sigma$,
\begin{equation}
\label{eq:dudnmess}
\tfrac{\partial u}{\partial n} =
i \kappa u - \left(i \kappa - \tfrac{\partial}{\partial n} \right) \left( D(u) - S(f) \right).
\end{equation}

\subsection{Variational Setting}
We let $\left( f, g \right) = \int_{\Omega^\prime} f(x) \overline{g(x)} dx$ be the standard $L^2$ inner product over the computational domain, and $\langle f , g \rangle_{S}$ that over some portion $S \subseteq \partial \Omega^\prime$ of its boundary.  We also let $H^s(\Omega^\prime)$ be the standard Sobolev spaces consisting of $L^2$ functions with weak derivatives of order up to and including $s$ in $L^2$.

When $X$ is some Banach space, $\| \cdot \|_X$ refers to its norm.  As we use several different norms throughout our analysis, we explicitly label each such norm to limit confusion.

We give a variational formulation of the PDE and hence a standard Galerkin finite element discretization as follows.
We take the inner product of~\eqref{eq:helmholtz} with any $v \in H^1(\Omega^\prime)$.  Integration by parts and the Neumann boundary condition on $\Gamma$ give
\begin{equation}
\label{eq:afteribp}
\left( \nabla u , \nabla v \right) - \kappa^2 \left( u , v \right)
- \langle \tfrac{\partial u}{\partial n}, v \rangle_\Sigma =
\langle f, v \rangle_{\Gamma},
\end{equation}
and substituting~\eqref{eq:dudnmess} in for $\tfrac{\partial u}{\partial n}$ on $\Sigma$ gives
\begin{equation}
\begin{split}
\left( \nabla u , \nabla v \right) - \kappa^2 \left( u , v \right)
- i \kappa \langle u, v \rangle_{\Sigma}
+ \langle \left( i \kappa - \tfrac{\partial}{\partial n}  \right) D(u) , v \rangle_\Sigma
\\= \langle f , v \rangle_\Gamma + \langle \left( i \kappa - \tfrac{\partial}{\partial n} \right) S(f), v \rangle_{\Sigma}.
\end{split}
\end{equation}
Hence, the solution to the Helmholtz equation~\eqref{eq:helmholtz} on $\Omega$ together with~\eqref{eq:neumann} and~\eqref{eq:sommerfeld} satisfies the variational problem of finding $u \in H^1(\Omega^\prime)$ such that
\begin{equation}
a(u, v) = F(v)
\end{equation}
for all $v \in H^1(\Omega^\prime)$.  Here, the bilinear form
\begin{equation}
  \label{eq:blf}
a(u, v)
= \left( \nabla u , \nabla v \right) - \kappa^2 \left( u , v \right)
- i \kappa \langle u, v \rangle_{\Sigma}
+ \langle \left( i \kappa - \tfrac{\partial}{\partial n} \right) D(u) , v \rangle_\Sigma
\end{equation}
consists of the standard bilinear form using transmission boundary conditions~\eqref{eq:transmission} augmented by nonlocal terms involving a convolution-type integral with a Green's function kernel.   We write
\( a(u, v) = a_L(u, v) + a_{NL}(u, v) \), where
\begin{equation}
\begin{split}
a_{L}(u, v) &= \left( \nabla u , \nabla v \right) - \kappa^2 \left( u , v \right)
- i \kappa \langle u, v \rangle_{\Sigma}, \\
a_{NL}(u, v) & = \langle \left( i \kappa - \tfrac{\partial}{\partial n} \right) D(u) , v \rangle_\Sigma.
\end{split}
\end{equation}

Similarly, the linear form
\begin{equation}
\label{eq:F}
F(v) = \langle f , v \rangle_\Gamma + \langle\left( i \kappa - \tfrac{\partial}{\partial n} \right) S(f), v \rangle_{\Sigma}
\end{equation}
involves the Neumann data on the scatterer together with its appearance in the single layer potential.

By taking $V_h \subset H^1(\Omega^\prime)$ as any suitable finite element space, we can introduce a Galerkin finite element method of finding $u_h \in V_h$ such that
\begin{equation}
a(u_h , v_h) = F(v_h)
\end{equation}
for all $v_h \in V_h$.

At this point, we pause compare our method to the Dirichlet-to-Neumann map $P$. (In the literature, the same operator is sometimes called the Steklov-Poincaré operator. Generically, S-P operators convert one type of boundary data into another.) Replacing $\tfrac{\partial u}{\partial n}$ on $\Sigma$ in~\eqref{eq:afteribp} with $P$ acting on $u$ would give
\[
\left( \nabla u , \nabla v \right) - \kappa^2 \left( u , v \right)
+ \langle Pu, v \rangle_\Sigma =
\langle f, v \rangle_{\Gamma}.
\]
Compared to~\eqref{eq:blf}, this appears to only have a single nonlocal term.  Moreover, $P$ is a symmetric elliptic operator from $H^{1/2}(\Sigma)$ into $H^{-1/2}(\Sigma)$, so that $\langle P u, u \rangle \geq 0$ and a G{\aa}rding estimate readily holds for the bilinear form. Unfortunately, the Steklov-Poincaré operator is not typically explicitly available, and thus its application requires the solution of a linear system at additional computational cost, e.g. in the form of a boundary integral equation solve. Approximating $Pu$ with a Fourier series is possible, however doing so requires separable geometry.

\subsection{Convergence theory}

Our argument will rely on showing the boundedness of the bilinear form $a$ and establishing a G{\aa}rding-type inequality.  Using standard techniques~\cite{brenner2007mathematical}, this leads to discrete solvability and optimal \emph{a priori} error estimates under a constraint on the maximal mesh size.

We will rely on the trace estimates~\cite{brenner2007mathematical,grisvard2011elliptic} that since $\Omega^\prime$ is Lipschitz, there exists a constant $C$ such that
\begin{equation}
\label{eq:trace}
\left\| v \right\|_{L^2(\partial \Omega^\prime)} \leq C \left\| v \right\|^{1/2}_{L^2(\Omega^\prime)} \left\| v \right\|^{1/2}_{H^1(\Omega^\prime)}
\leq C \left\| v \right\|_{H^1(\Omega^\prime)}
\end{equation}
for all $v \in H^1(\Omega^\prime)$.

\begin{proposition}
If the Neumann data satisfies $f \in H^{-\frac{1}{2}}(\Gamma)$, the functional $F$ defined in~\eqref{eq:F} is a bounded linear functional on $H^1$.
\end{proposition}
\begin{proof}
Linearity is clear from the linearity of integration and
differentiation.  To see that it is bounded, let $v \in H^1(\Omega)$
be given.  The local portion of $F$ is bounded thanks to
Cauchy-Schwarz and the second trace estimate in~\eqref{eq:trace}.  For
the nonlocal portion, it is known~\cite{steinbach2007numerical} that
$S(f) \in H^1(\Omega^\prime)$ and so it has a normal derivative on $\Sigma$
in $H^{-1/2}(\Sigma)$.
\end{proof}

The following result implies both the boundedness of $a_{NL}$ on $H^1 \times H^1$ and is critical to establishing the G{\aa}rding inequality:
\begin{lemma}
\label{nllemma}
There exists a $C_{NL} > 0$ such that for all $u, v \in H^1(\Omega^\prime)$,
\begin{equation}
|a_{NL}(u, v)| \leq C_{NL} \left\| u \right\|_{L^2(\Gamma)} \left\| v \right\|_{L^2(\Sigma)}.
\end{equation}
\end{lemma}
\begin{proof}
First, we simplify the notation by writing the first argument in $a_{NL}$ as
\begin{equation}
\label{eq:knl}
 \left( i \kappa - \tfrac{\partial}{\partial n} \right) D(u)
= \left( i \kappa - \tfrac{\partial}{\partial n} \right) \int_{\Gamma} \mathcal{K}(x-y) u(y) dy,
\end{equation}
From the properties of the kernel, $\mathcal{K}(x-y)$ is smooth and bounded provided that $\| x - y \|$ is bounded below away from zero.  Since we have $x \in \Sigma$ and $y \in \Gamma$, this is the case as long as the truncating boundary stays away from the scatterer.  By writing the normal derivative in~\eqref{eq:knl} as the limit of a difference quotient, passing under the integral, and appealing to the Lebesgue Dominated Convergence Theorem in the usual way, we can then write
\begin{equation}
\left( i \kappa - \tfrac{\partial}{\partial n} \right) D(u)
= \int_\Gamma \left( i \kappa - \tfrac{\partial}{\partial n} \right) \mathcal{K}(x-y) dy
= \int_\Gamma \widetilde{\mathcal{K}}(x-y) dy,
\end{equation}
where $\widetilde{\mathcal{K}}(x-y)$ is also smooth and bounded for $x$ and $y$ separated.  We can write the nonlocal bilinear form now as
\begin{equation}
|a_{NL}(u, v)|
= \left| \int_{\Gamma}\int_{\Sigma} \widetilde{\mathcal{K}}(x, y) u(y) \overline{v(x)} dx \, dy \right|
\leq
K_0 \int_\Gamma u(y) dy \int_{\Sigma} \overline{v(x)} dx,
\end{equation}
and the result holds with with $C_{NL} = K_0 | \Gamma|^{1/2} |\Sigma|^{1/2}$ by the Cauchy-Schwarz inequality.
\end{proof}

\begin{proposition}
There exists $C_B > 0$ such that for all $u, v \in H^1(\Omega^\prime)$,
\begin{equation}
\left| a(u, v) \right| \leq C_B \left\| u \right\|_{H^1(\Omega^\prime)} \left\| v \right\|_{H^1(\Omega^\prime)}.
\end{equation}
\end{proposition}
\begin{proof}
Let $u, v \in H^1(\Omega^\prime)$.  Then
\begin{equation}
  \begin{split}
|a(u, v)| \leq & 
\left\| \nabla u \right\|_{L^2(\Omega^\prime)} \left\| \nabla v \right\|_{L^2(\Omega^\prime)} \\
& + \kappa^2 \left\| u \right\|_{L^2(\Omega^\prime)} \left\| v \right\|_{L^2(\Omega^\prime)}
+ \kappa \left\| u \right\|_{L^2(\Sigma)} \left\| v \right\|_{L^2(\Sigma)} + |a_{NL}(u, v)|,
\end{split}
\end{equation}
and the proof is finished by applying the previous Lemma and trace theorem.
\end{proof}

The bilinear form $a$ satisfies a G{\aa}rding inequality.  That is, shifting $a$ by a multiple of the $L^2$ inner product renders a coercive bilinear form.  For complex Hilbert spaces, it is sufficient to demonstrate that the real part itself is coercive.

\begin{proposition}
There exists a real number $M$ and an $\alpha > 0$ such that
\begin{equation}
\operatorname{Re}(a(u,u)) + M \left\| u \right\|_{L^2(\Omega^\prime)}^2 \geq \alpha \left\| u \right\|_{H^1(\Omega^\prime)}^2.
\end{equation}
\end{proposition}
\begin{proof}
We calculate
\begin{equation}
\label{eq:garding}
a(u, u)
= \left\| \nabla u \right\|_{L^2(\Omega^\prime)}^2 - \kappa^2 \left\| u \right\|_{L^2(\Omega^\prime)}^2
- i \kappa \left\| u \right\|_{L^2(\Sigma)}^2 + a_{NL}(u, u),
\end{equation}
and note that the real part of this is just
\begin{equation}
\begin{split}
\operatorname{Re}(a(u,u))
& = \left\| \nabla u \right\|_{L^2(\Omega^\prime)}^2 - \kappa^2 \left\| u \right\|_{L^2(\Omega^\prime)}^2 + \operatorname{Re}(a_{NL}(u, u)) \\
& = \left \| u \right\|_{H^1(\Omega^\prime)}^2 - \left( \kappa^2 + 1 \right) \left\| u \right\|_{L^2(\Omega^\prime)}^2  + \operatorname{Re}(a_{NL}(u, u)).
\end{split}
\end{equation}
Using Lemma~\ref{nllemma}, the trace inequality~\eqref{eq:trace}, and a weighted Young's inequality, we can bound this below by
\begin{equation}
\begin{split}
\operatorname{Re}(a(u,u))
& \geq  \left \| u \right\|_{H^1(\Omega^\prime)}^2
- (\kappa^2 +1) \left\| u \right\|_{L^2(\Omega^\prime)}^2
  - C_{NL} \left\| u \right\|_{L^2(\Sigma)} \left\| u \right\|_{L^2(\Gamma)} \\
  & \geq  \left \| u \right\|_{H^1(\Omega^\prime)}^2
  - (\kappa^2+1) \left\| u \right\|_{L^2(\Omega^\prime)}^2 - C_{NL} \left\| u \right\|_{H^1(\Omega^\prime)} \left\| u \right\|_{L^2(\Omega^\prime)} \\
& \geq \tfrac{1}{2} \left\| u \right\|_{H^1(\Omega^\prime)}^2 - \left(\kappa^2 + 1 + \tfrac{C_{NL}^2}{2} \right) \left\| u \right\|_{L^2(\Omega^\prime)}^2,
\end{split}
\end{equation}
so~\eqref{eq:garding} holds with $\alpha = \tfrac{1}{2}$ provided $M \geq \kappa^2 + \tfrac{3+C_{NL}^2}{2}$.
\end{proof}

Now, following standard techniques for general elliptic (but possibly not coercive) problems~\cite{brenner2007mathematical}, suitably adapted for the complex-valued case, we have a general solvability and approximation result.  We suppose that the standard abstract approximation result
\begin{equation}
\inf_{v \in V_h} \left\| u - v \right\|_{H^1(\Omega^\prime)} \leq C_A h \left| u \right|_{H^2(\Omega^\prime)}
\end{equation}
holds and that the solution to~\eqref{eq:helmholtz} is in $H^2(\Omega^\prime)$.  We also require that the adjoint problem of finding $w \in H^1(\Omega^\prime)$ such that
\begin{equation}
a(v, w) = \left(f , v \right)
\end{equation}
for all $v \in H^1(\Omega^\prime)$ has a unique solution with regularity estimate
\begin{equation}
\left| u \right|_{H^2(\Omega^\prime)} \leq C_R \left\| f \right\|_{L^2(\Omega^\prime)}.
\end{equation}
With these assumptions, the arguments leading to Theorem 5.7.6 of~\cite{brenner2007mathematical} give this result:
\begin{theorem}
  \label{convergence}
Under the above conditions, there exists $h_0$ such that for $h \leq h_0$,
the discrete variational problem~\eqref{eq:h1err} has a unique solution $u_h$ satisfying the error estimate
\begin{equation}
\label{eq:h1err}
\left\| u - u_h \right\|_{H^1(\Omega^\prime)} \leq C \inf_{v \in V_h} \left\| u - v \right\|_{H^1(\Omega^\prime)}.
\end{equation}
Moreover, with the same assumptions, there exists another $C > 0$ such that
\begin{equation}
\label{eq:l2err}
\left\| u - u_h \right\|_{L^2(\Omega^\prime)} \leq C h \left\| u - u_h \right\|_{H^1(\Omega^\prime)}.
\end{equation}
\end{theorem}
Note that this is a quasi-optimal result, independent of the particular choice of polynomial spaces.  So, it gives error estimates for higher-order approximations as well as for standard $P^1$ elements.

\begin{remark}
In particular, following~\cite{brenner2007mathematical}, one can show
\begin{equation}
h_0 = \frac{\left(\tfrac{\alpha}{2M}\right)^{1/2}}{C_B C_A C_R},
\end{equation}
and the constant $C$ in~\eqref{eq:h1err} can be taken as
$ \tfrac{2C_B}{\alpha}$ and that in~\eqref{eq:l2err} as $C_B C_A C_R$.
\end{remark}

\begin{remark}
This convergence theory assumes the layer potentials and boundary integrals are evaluated exactly.  These results can be
extended to account for approximation to the layer potential along
the lines of \cite[Thm.~13.6/7]{kress_linear_1999} and quadrature
in the bilinear forms using the standard theory of variational
crimes~\cite{brenner2007mathematical}.
\end{remark}

\section{Linear algebra}
\label{sec:solve}
We can effectively solve our variational formulation using preconditioned GMRES, which is a parameter-free algorithm approximating the solution of a linear system $A \bfx = \bfb$ as the element of the Krylov subspace $\operatorname{span}\{ A^i \bfb \}_{i=0}^m$ minimizing the equation residual.  Building the subspace does not require the entries of $A$, just the action of $A$ on vectors.  Unlike conjugate gradients, GMRES is not restricted to operators that are symmetric and positive definite.

For most problems arising in the discretization of PDE, the condition number of $A$ degrades quickly under mesh refinement, and GMRES is most frequently used in conjunction with a (left) \emph{preconditioner}.  Mathematically, we multiply the linear system through by some matrix $\widehat{P}^{-1}$:
\begin{equation}
\widehat{P}^{-1} A \bfx = \widehat{P}^{-1} \bfb,
\end{equation}
and so the Krylov space then is $\operatorname{span}\{ \left( \widehat{P}^{-1}A \right)^i \widehat{P}^{-1} \bfb \}_{i=0}^m$.

The overall performance of GMRES typically is determined by two factors -- the cost of building and applying the operators $\widehat{P}^{-1}$ and $A$, and the total number of iterations.  One hopes to obtain a per-application cost that scales linearly (or log-linearly) with respect to the number of unknowns in the linear system, and a total number of GMRES iterations that is bounded independently of the number of unknowns.  We think of $\widehat{P}^{-1}$ being an approximation to the inverse of some matrix $P$ that somehow approximates $A$.
In our case, we will find it useful to let $P = A^L$, the local part of the operator.  Then, applying $\widehat{P}^{-1}$ might correspond to applying the inverse of $P$ by a sparse direct method or perhaps just some sweeps of a multigrid algorithm.

\subsection{Structure of the discrete problem}%
\label{sec:disc}
By taking a standard finite element basis $\{\phi_i\}_{i=1}^N$ for $V_h$, the stiffness matrix is
\begin{equation}
\label{eq:stiffness}
A_{ij} = a(\phi_j, \phi_i) = a_{L}(\phi_j, \phi_i) + a_{NL}(\phi_j, \phi_i)
= A^L_{ij} + A^{NL}_{ij}.
\end{equation}
The portion $A^{L}$ is the standard sparse matrix one obtains for discretization of the Helmholtz operator with transmission boundary conditions, while $A^{NL}$ contains the contributions for the nonlocal terms.
To further consider the sparsity of this system, supposing we use standard $P^1$ basis functions and have about $\mathcal{O}(N^d)$ total vertices and hence basis functions. Then $A^L$ has nonzero entries corresponding to vertices sharing a common mesh element -- typically about 6-7 nonzeros per row on two-dimensional triangulations and 20-30 for three-dimensional tetrahedral meshes when using linear basis functions.  The total storage required for $A^L$ will be proportional to the number of vertices in the mesh.

Explicit sparse storage of $A^{NL}$, however, can be quite different.
Since $K$ involves a convolution-type integral over $\Gamma$,
\begin{equation}
A^{NL}_{ij} = \langle\left(i \kappa -  \tfrac{\partial}{\partial n}\right) K(\phi_j), \phi_i \rangle_\Sigma
\end{equation}
will be nonzero whenever $\phi_j$ is supported on $\Gamma$ and $\phi_i$ is supported on $\Sigma$.  Suppose that we have $\mathcal{O}(N^{d-1})$ basis functions supported on $\Sigma$ and the same order on $\Gamma$.  Then, $A^{NL}$ will be nonzero except for a dense logical subblock.  However, each basis function associated with $\Sigma$ will interact with each basis functions associated with $\Gamma$, so that the dense subblock will contain about $\mathcal{O}(N^{2d-2})$ nonzero entries.  When $d=2$, $A^{NL}$ has the same order of nonzeros as $A^L$ and so conceivably could be stored explicitly.  On the other hand, when $d=3$, $A^{NL}$ has $\mathcal{O}(N^4)$ nonzero entries and so its storage dominates that of the local part $A^L$.  Consequently, a matrix-free application of $A^{NL}$ that bypasses the storage may be preferred, as described in Section~\ref{sec:pytential}.

\subsection{Operator application}
From~\eqref{eq:stiffness}, the system matrix $A=A^L + A^{NL}$ is the sum of two matrices corresponding to the local and nonlocal terms in the bilinear form.  Although we could implement a matrix-free action of $A^L$, we opt to assemble a standard sparse matrix and only apply $A^{NL}$ in a matrix-free fashion as follows.

Recall that $A^{NL} = a_{NL}(\phi_j, \phi_i) = \langle \left(i\kappa - \tfrac{\partial}{\partial n} \right)D(\phi_j), \phi_i \rangle$.  Any vector $\bfx \in \mathbb{R}^{\dim V_h}$ can be identified uniquely with some $v_h$ in the finite element space so that
\begin{equation}
\left( A^{NL} \bfx \right)_i
= a_{NL}(v_h, \phi_i).
\end{equation}
Note that $(A^{NL} \bfx)_i$ will be nonzero exactly when $\phi_i$ has support on the exterior boundary $\Sigma$.

In a startup phase, we prepare the boundary geometry according to the algorithm of~\cite{wala_fast_2019} construct a GIGAQBX tree structure~\cite{rachh_fast_2017, wala_fast_2018, wala_fast_2019,  wala_approximation_2020} for the approximation of the layer potentials $D$ and $S$.  These allow us to efficiently approximate $\left( i \kappa - \tfrac{\partial}{\partial n} \right) D(v_h)$ at a collection of `target' points.  In particular, we can evaluate on quadrature points on each facet of $\Sigma$.  Hence, we can loop over the facets on $\Sigma$ to integrate against the basis functions supported on that facet and sum the contributions in the usual way.  This gives the action of $A^{NL}$ onto some $\bfx$, and the full action of $A$ onto $\bfx$ is computed by summing this with $A^L \bfx$ computed by a standard sparse matrix-vector product.

\subsection{Preconditioners}
Rather than letting the preconditioning matrix $P$ equal $A$ itself, we opt for $P = A^L$.  If we were to
exactly invert $A^L$, then the resulting system becomes
\begin{equation}
\left( I + \left(A^L\right)^{-1} A^{NL} \right) x = \left( A^{L} \right)^{-1} b.
\end{equation}
Since $A^L$ discretizes an elliptic equation and $A^{NL}$ a bounded operator, this has the form of a discretization of a compact perturbation of the identity.  In~\cite{gmati2007comments}, GMRES convergence for a similar situation was shown to be very favorable.  We also comment that preconditioning a system to obtain a compact perturbation of the identity was used heuristically to good effect for B\'enard convection~\cite{howle2012block}.

It is possible to replace the inverse of $A^L$ with an approximation, and it is likewise possible to use a suitable, spectrally equivalent preconditioner, such as
algebraic multigrid~\cite{adams2004algebraic, ruge1987algebraic, olson2010smoothed}.  This gives a preconditioner that scales well with mesh refinement, but can degrade as the wave number increases~\cite{ernst2012difficult, gander2015applying}.

\section{Implementation}
\label{sec:imp}
Our implementation rests on combining the capabilities of \firedrake~\cite{firedrake} for the finite element part of our problem and \pytential{}~\cite{pytential} for the evaluation of layer potentials $K$ and $S$.   Krylov solvers and preconditioners are accessed using \petsc.

\subsection{Firedrake}
\firedrake~\cite{firedrake} is an automated system for the solution of partial differential equations using the finite element method.  It allows users to describe the variational form of a PDE using the Unified Form Language~\cite{Alnaes:2014}, from which it generates effective lower-level numerical code.  We make use of \firedrake{} for loading computational meshes, defining the local part of $a$, and integrating evaluated layer potentials against test functions.   \firedrake{} can be built supporting complex arithmetic at every level (definition of bilinear forms down to a complex-enabled \petsc{} build).

\firedrake{} also makes it possible to compare our new method against
domain truncation by means of a perfectly matched layer (PML).
We implement the technique of~\cite{Bermudez:2006}~which uses an
unbounded integral as the absorbing function on the PML.  This approach is parameter-free and simple to implement in UFL.

\subsection{Pytential}
\label{sec:pytential}
\pytential{}~\cite{pytential} is an open-source, MIT licensed software system that allows the
evaluation of layer potentials from source geometry represented by
unstructured meshes in two and three dimensions with near-optimal complexity
and at a high order of accuracy. The main aspects of functionality
provided by \pytential{} are the discretization of a source surface
using discretization tools (through its use of a sister tool,
\software{meshmode}~\cite{meshmode}) for high-order accurate nonsingular
quadrature~\cite{xiao_numerical_2010}, its refinement according to
accuracy requirements~\cite{wala_fast_2019},
and, finally, the evaluation of weakly singular, singular, and
hypersingular integral operators via quadrature by expansion
(QBX)~\cite{kloeckner_quadrature_2013} and the associated GIGAQBX fast
algorithm~\cite{wala_fast_2018}, with rigorous accuracy guarantees in
two and three dimensions~\cite{wala_approximation_2020}. This fast
algorithm, can, in turn make use of \software{FMMLIB}~\cite{gimbutas_fast_2009,pyfmmlib}
for the evaluation of translation operators in the moderate-frequency
regime for the Helmholtz equation.

While the layer potential evaluations in $a_{NL}$ are nonsingular, we
nonetheless benefit from the use of the QBX machinery in the event
that source and target surfaces are chosen to lie in close proximity
for increased efficiency of the finite element method. See~\cite{wala_fast_2019}
for estimates of the error incurred in the evaluation of the layer
potential.
Our 2D experiments employ higher order discretizations and fine meshes,
so we set FMM order adaptively at each level of the FMM tree
to provide a precision stricter than machine epsilon for double precision.
Our 3D experiments are limited to piecewise linear elements on
coarse meshes, so it is sufficient to use an FMM order of 12.

\subsection{Representing the linear system in \petsc{}}
At the top level, our code builds and solves the linear system \eqref{eq:stiffness}.  To do this, we have implemented a \software{Python} matrix type in \software{petsc4py}~\cite{Dalcin:2011}.  Its Python context builds the bilinear form $a_{L}$ in \firedrake{} and assembles $A^{L}$.  It also sets up the layer potential evaluation in \pytential{}.  The class also provides a multiplication method that multiplies by $A^L$ (itself just a \petsc{} call) and $A^{NL}$ (which requires more code) and sums the results.  It also provides a handle to $A^L$ so that it can be used as a preconditioning matrix.  Setting up a KSP context in \petsc{}, we can then select from any available Krylov method and apply any preconditioning technique to $A^L$ in the standard ways.

The application of $A^{NL}$ to a vector requires some low-level interaction of \firedrake{} and \pytential{} beneath their public interfaces, and warrants some explanation. Data transfer between \pytential{} and \firedrake{} occurs in two directions. The transfer of density information from \firedrake{} to \pytential{} occurs through (exact) interpolation within $P^N$ from the $C^0$ finite element space used for $a_L$ to the discontinuous finite element space on Vioreanu-Rokhlin nodes~\cite{vioreanu_spectra_2014} used for the density in $a_{NL}$. This requires some attention to details regarding ordering of degrees of freedom, vertices~\cite{RognesKirbyEtAl2009a}, and data formats. The transfer of layer potential information back to \firedrake{} meanwhile is straightforward by comparison. \pytential{} is able to evaluate the layer potential with guaranteed accuracy \emph{anywhere} in the target domain, even close to the source surface, where this might otherwise require special treatment such as near-singular quadrature, e.g. by adaptive techniques. Thus we merely evaluate the layer potential at a set of quadrature points supplied by \firedrake{} to obtain an approximate projection of the (analytically) $C^\infty$ potential back into the $C^0$ finite element space.

\section{Numerical results}
\label{sec:results}
Now, we present some empirical investigation of our method.  We establish the accuracy obtained using finite element approximation using our nonlocal boundary condition and also consider preconditioning the nonlocal boundary system.  We find that the accuracy obtained using the nonlocal boundary condition compares favorably with that rendered by PML and the transmission boundary condition.  Moreover, when methods are available to accurately approximate the inverse of $A^L$, we find that it is an excellent preconditioner for the overall system.  However, as the wave number increases, the difficulty of attaining an accurate approximation increases.

To verify the accuracy of our method in two and three space dimensions, we chose the unit disc/sphere as a scatterer and use a manufactured solution based on the free-space Helmholtz Green's function.  That is, the true solution outside of the scatterer is taken as
\begin{equation*}
    \begin{cases}
        u(x) = \frac{i}{4} H_0^{(1)}(\kappa |x|) & d = 2, \\
        u(x) = \frac{i}{4\pi|x|} e^{i\kappa |x|} & d = 3.
    \end{cases}
\end{equation*}
In two dimensions, we truncated the domain as the $6\times6$ square centered at the origin, shown in Figure~\ref{fig:2d_pml_colored} with the PML region separately highlighted.  In three dimensions, we created an analogous mesh, the unit sphere embedded in $[-3,3]^3$, with the PML sponge region taking up $[-3,3]^3\setminus[-2,2]^3$, shown in Figure~\ref{fig:3dmesh}.

\begin{figure}[!ht]
    \begin{center}
        \includegraphics[width=0.8\textwidth]{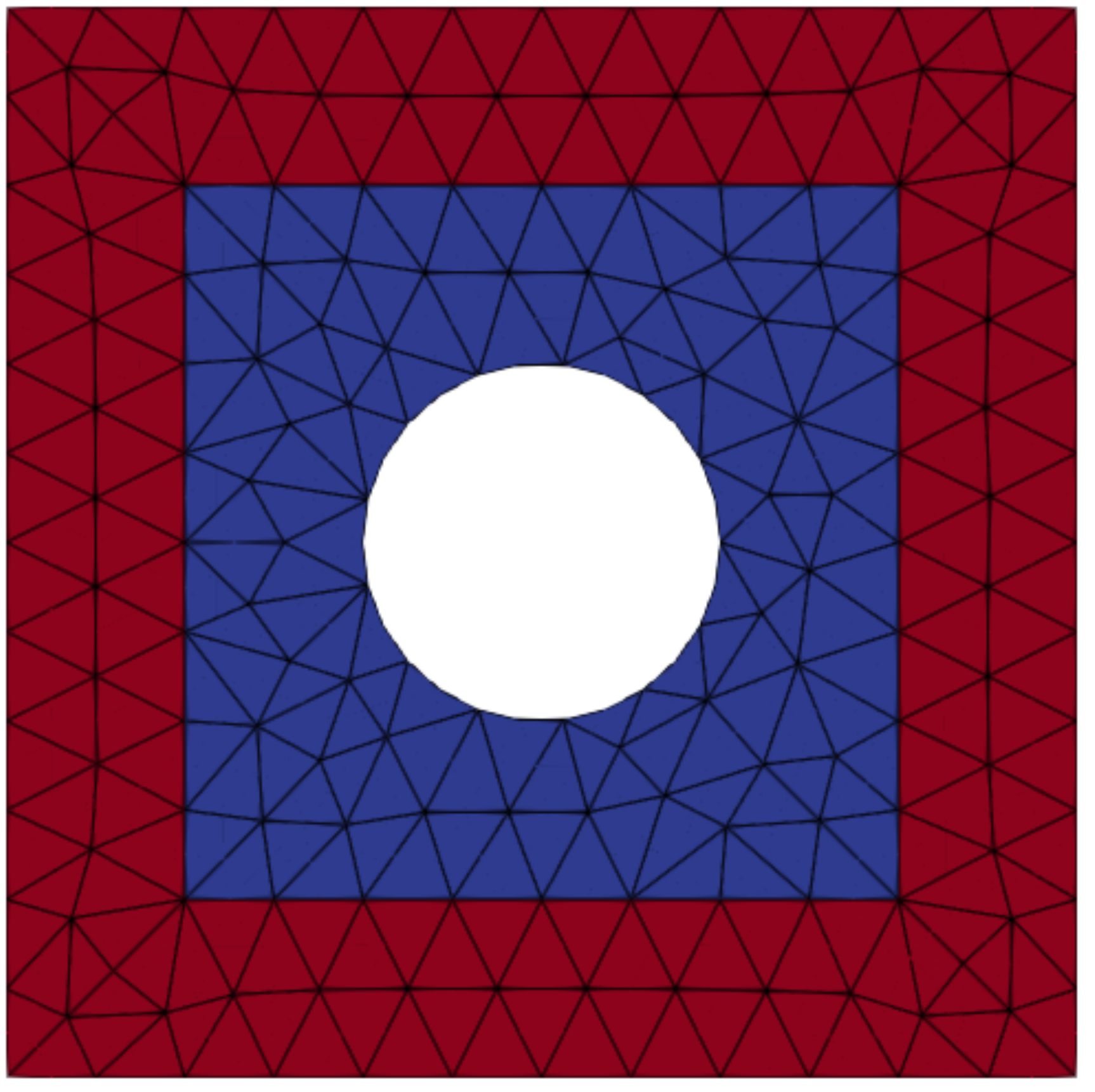}
    \end{center}
    \caption{Example 2d mesh (with PML region colored in red)}
    \label{fig:2d_pml_colored}
\end{figure}

\begin{figure}[!ht]
    \includegraphics[width=0.9\textwidth]{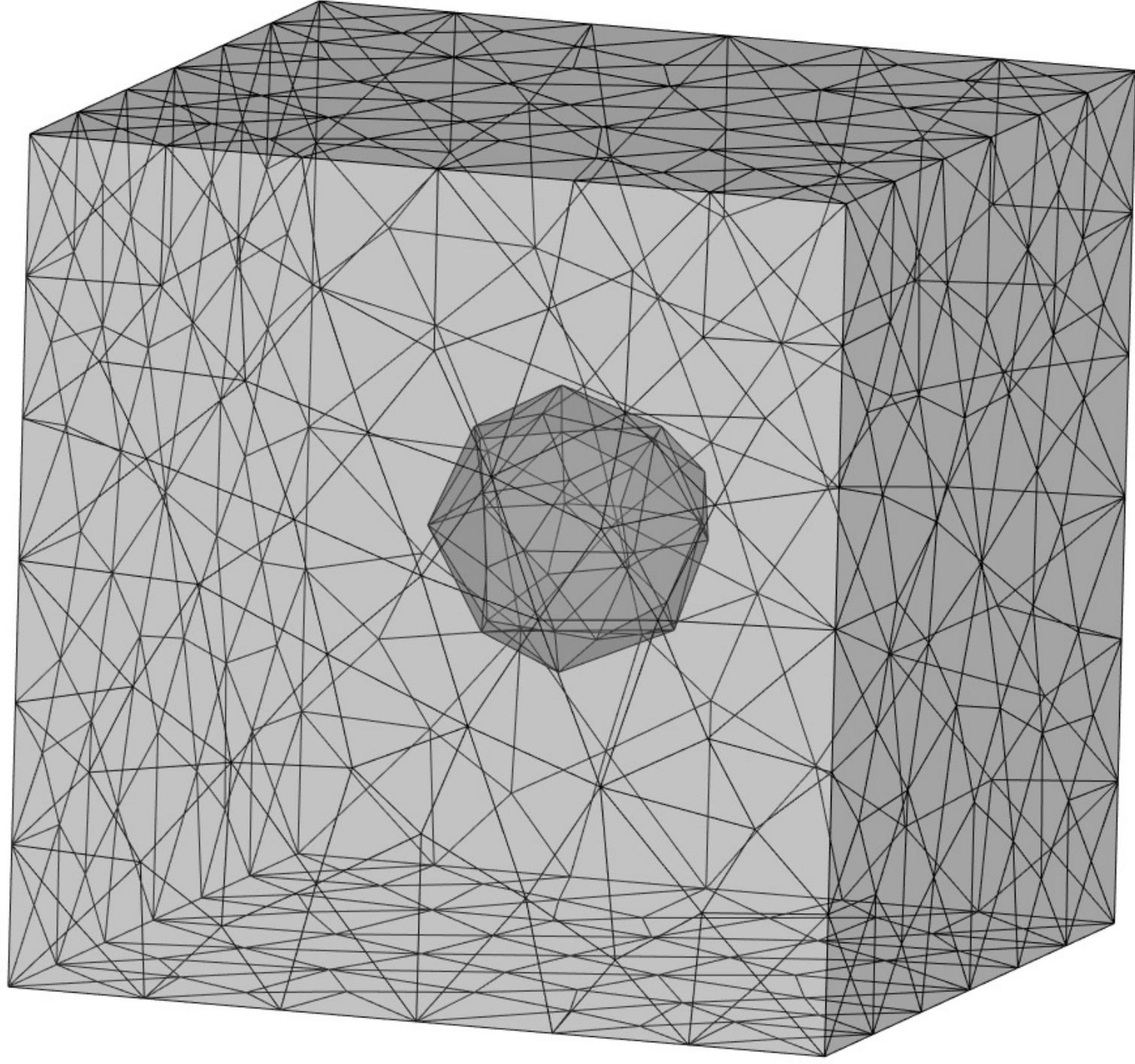}
    \caption{Sample coarse 3D mesh with a spherical exclusion at the center of a cube.}
    \label{fig:3dmesh}
\end{figure}

We compare our approach with transmission boundary conditions
and with PML.
Since the PML-based simulation is only accurate in the non-PML region (the blue region in
Figure~\ref{fig:2d_pml_colored}),
we evaluate the error only over the non-PML region for all
methods, even though we solve over the entire computational domain.
For computations with piecewise linear approximating spaces, we use affine geometry, although we use quadratic geometry for computations with higher-order spaces.

We implement both PML and transmission boundary
conditions in \firedrake.
For PML, we use unbounded absorbing
functions as described in~\cite{Bermudez:2006}.
This form of PML has no parameters to be fitted,
is simple to implement in \firedrake, and
has been shown to recover the exact
solution (up to discretization errors) on annular domains~\cite{Bermudez2007}.
However, other implementations of PML (or equivalent boundary
conditions) can obtain higher
accuracy at lower wavelengths~\cite{Druskin2016, Hagstrom2019}.

We compare the accuracy of transmission boundary conditions, PML,
transmission, and our new approach in 2D for degree 1
approximations
in Figure~\ref{fig:2ddegree1accuracy}.  We observe that the transmission boundary
conditions, which incur a perturbation of the PDE, lead to convergence
to a slightly incorrect solution.  Both PML and our new boundary
conditions, however, seem to be converging to the true solution at the
proper rate of $\mathcal{O}(h^2)$ predicted in
Theorem~\ref{convergence}.  For small $\kappa$, the nonlocal condition
seems quite a bit more accurate, although they give nearly the same
error for larger $\kappa$.  Accuracy for the 2D case is reported
for higher degrees  in
Figures~\ref{fig:2ddegree2accuracy},~\ref{fig:2ddegree3accuracy}, and~\ref{fig:2ddegree4accuracy} respectively.  In these case, we observe the theoretically-predicted convergence rates for our nonlocal method, although for higher degree our (suboptimal) PML does not obtain full accuracy.  Our theoretical results apply to 3D as well as 2D, although the computations are considerably more expensive.  As a simple test, we have used linear polynomials on rather coarse meshes, presenting the results in
Figure~\ref{fig:3ddegree1accuracy}.  We see comparable behavior to that obtained in 2D.

In Figure~\ref{fig:accuracyVSdofs}, we study the error obtained versus the number of degrees of freedom using various orders of approximation.  In these computations, we use the domain $\Omega^\prime$ as $[-2,2]^2$ minus the unit square and pose the boundary condition~\eqref{eq:dudnmess} along the outer boundary.  Then, for each $\kappa \in \{ 0.1, 1, 5, 10 \}$ and polynomial degrees 1 through 4, we computed the $L^2$ error in the numerical approximation.  In addition to giving faster convergence rates, we also see that higher-order polynomials provide a lower error per degree of freedom used.  Moreover, no modifications to our method or boundary condition were required to obtain this higher accuracy.
 
%%%%%%%%%%%%%%%%%%%%%%%%                                   %%%%%%%%%%%%%%%%%%%%%%%%
%%%%%%%%%%%%%%%%%%%%%%%  Template for single accuracy plot  %%%%%%%%%%%%%%%%%%%%%%%%
%%%%%%%%%%%%%%%%%%%%%%%%                                   %%%%%%%%%%%%%%%%%%%%%%%%
% Usage: either
% plotaccuracy[kappa](color)[dimension](degree)[Relative ]
%        or
% plotaccuracy[kappa](color)[dimension](degree)[]
\def \plotaccuracy[#1](#2)[#3](#4)[#5]
{
    \begin{tikzpicture}[scale=0.65]
    \begin{loglogaxis}[legend style={at={(1.0,0.0)},anchor=south east},
                       xlabel=$h$,
                       ylabel=#5$L^2$ Error,
                       title={$\kappa=$#1, degree$=$#4},
                       cycle list={
                           {#2, mark=x}, {#2, mark=o}, {#2, mark=+},
                       },
                       filter discard warning=false,
                       scale=0.8]
        \foreach \method in {pml,transmission,nonlocal} {
            \addplot+[only marks, thick]
                table [x=h,
                       y
                       expr={\thisrow{kappa}==#1?\thisrow{\method\space lu #5L2 Error}:nan},
                       col sep=comma] {#3d_data_degree#4.csv};
            \addlegendentryexpanded{\method};
        };
    \end{loglogaxis}
    \end{tikzpicture}
}
%%%%%%%%%%%%%%%%%%%%%%%%                                             %%%%%%%%%%%%%%%%%%%%%%%%
%%%%%%%%%%%%%%%%%%%%%%%  Template to plot accuracy given dim, degree  %%%%%%%%%%%%%%%%%%%%%%%%
%%%%%%%%%%%%%%%%%%%%%%%%                                             %%%%%%%%%%%%%%%%%%%%%%%%
% USAGE: 
% \plotAccuracyOverKappa[dim](degree)[Relative ]
%   or
% \plotAccuracyOverKappa[dim](degree)[]
\def \plotAccuracyOverKappa[#1](#2)[#3]
{
    \begin{center}
        \begin{tabular}{c c}
            \plotaccuracy[0.1](red)[#1](#2)[#3] & \plotaccuracy[1.0](blue)[#1](#2)[#3] \\
            \plotaccuracy[5.0](brown)[#1](#2)[#3] & \plotaccuracy[10.0](black)[#1](#2)[#3] \\
        \end{tabular}
     \end{center}
 }
%%%%%%%%%%%%%%%%%%%%%%%%                      %%%%%%%%%%%%%%%%%%%%%%%%
%%%%%%%%%%%%%%%%%%%%%%%  2d degree 1 accuracy  %%%%%%%%%%%%%%%%%%%%%%%%
%%%%%%%%%%%%%%%%%%%%%%%%                      %%%%%%%%%%%%%%%%%%%%%%%%
 \begin{figure}[!ht]
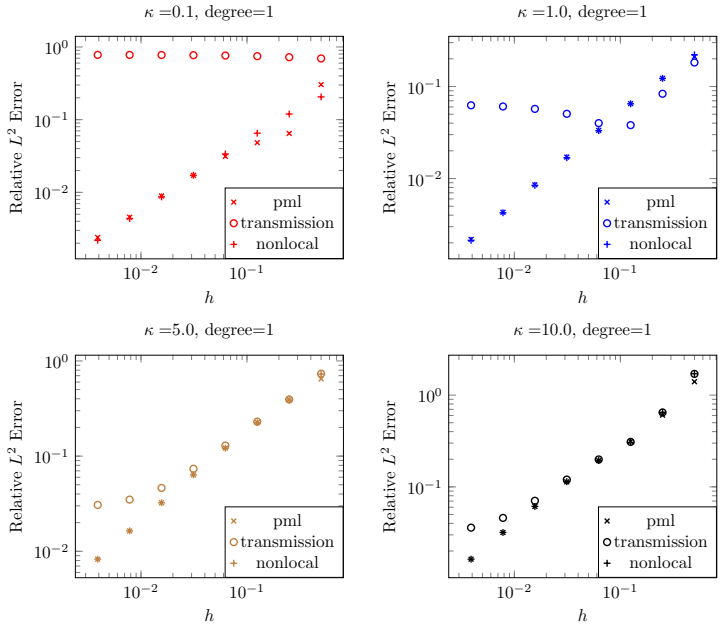

    \plotAccuracyOverKappa[2](1)[Relative ]
\caption{$L^2$ Relative error using degree 1 polynomials with respect to refinement in 2D. 
A KSP relative tolerance of $10^{-12}$ was used.
Since the transmission BC are a perturbation of the actual boundary value problem, the convergence levels off as the method converges to a slightly incorrect solution.  The PML and nonlocal methods give comparable accuracy.}
~\label{fig:2ddegree1accuracy} \end{figure}
%%%%%%%%%%%%%%%%%%%%%%%%                      %%%%%%%%%%%%%%%%%%%%%%%%
%%%%%%%%%%%%%%%%%%%%%%%  2d degree 2 accuracy  %%%%%%%%%%%%%%%%%%%%%%%%
%%%%%%%%%%%%%%%%%%%%%%%%                      %%%%%%%%%%%%%%%%%%%%%%%%
\begin{figure}[!ht]
    \plotAccuracyOverKappa[2](2)[Relative ]
\caption{$L^2$ Relative error using degree 2 polynomials with respect to refinement in 2D.
A KSP relative tolerance of $10^{-12}$ was used.
}
\label{fig:2ddegree2accuracy} \end{figure}
%%%%%%%%%%%%%%%%%%%%%%%%                      %%%%%%%%%%%%%%%%%%%%%%%%
%%%%%%%%%%%%%%%%%%%%%%%  2d degree 3 accuracy  %%%%%%%%%%%%%%%%%%%%%%%%
%%%%%%%%%%%%%%%%%%%%%%%%                      %%%%%%%%%%%%%%%%%%%%%%%%
\begin{figure}[!ht]
    \plotAccuracyOverKappa[2](3)[Relative ]
\caption{$L^2$ Relative error using degree 3 polynomials with respect to refinement in 2D.
A KSP relative tolerance of $10^{-12}$ was used.
}
\label{fig:2ddegree3accuracy} \end{figure}
%%%%%%%%%%%%%%%%%%%%%%%%                      %%%%%%%%%%%%%%%%%%%%%%%%
%%%%%%%%%%%%%%%%%%%%%%%  2d degree 4 accuracy  %%%%%%%%%%%%%%%%%%%%%%%%
%%%%%%%%%%%%%%%%%%%%%%%%                      %%%%%%%%%%%%%%%%%%%%%%%%
\begin{figure}[!ht]
    \plotAccuracyOverKappa[2](4)[Relative ]
\caption{$L^2$ Relative error using degree 4 polynomials with respect to refinement in 2D.
A KSP relative tolerance of $10^{-12}$ was used.
}
\label{fig:2ddegree4accuracy} \end{figure}
We recall that our theoretical results apply to 3D as well as 2D, although the computations are considerably more expensive.  As a simple test, we have used linear polynomials on rather coarse meshes, presenting the results in
Figure~\ref{fig:3ddegree1accuracy}.  We see comparable behavior to that obtained in 2D.

In Figure~\ref{fig:accuracyVSdofs}, we study the error obtained versus the number of degrees of freedom using various orders of approximation.  In these computations, we use the domain $\Omega^\prime$ as $[-2,2]^2$ minus the unit square and pose the boundary condition~\eqref{eq:dudnmess} along the outer boundary.  Then, for each $\kappa \in \{ 0.1, 1, 5, 10 \}$ and polynomial degrees 1 through 4, we computed the $L^2$ error in the numerical approximation.  In addition to giving faster convergence rates, we also see that higher-order polynomials provide a lower error per degree of freedom used.  Moreover, no modifications to our method or boundary condition were required to obtain this higher accuracy.
%%%%%%%%%%%%%%%%%%%%%%%%                      %%%%%%%%%%%%%%%%%%%%%%%%
%%%%%%%%%%%%%%%%%%%%%%%  3d degree 1 accuracy  %%%%%%%%%%%%%%%%%%%%%%%%
%%%%%%%%%%%%%%%%%%%%%%%%                      %%%%%%%%%%%%%%%%%%%%%%%%
\begin{figure}[!ht]
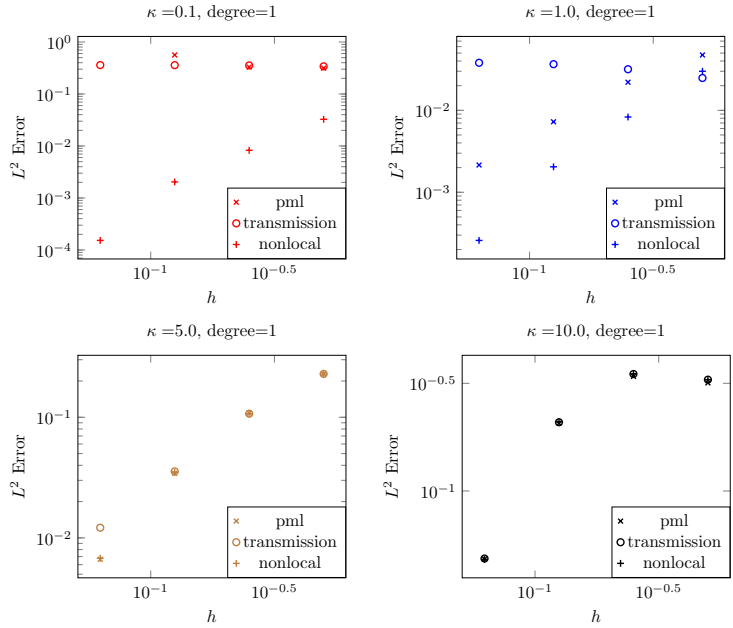

    \plotAccuracyOverKappa[3](1)[]
\caption{$L^2$ Error with respect to refinement in 3D using degree 1
    polynomials.
A KSP relative tolerance of $10^{-7}$ was used.
    Comparable
results to Figure~\ref{fig:2ddegree1accuracy} are obtained, although we have not been able to attain the same mesh resolutions as in 2D.}
\label{fig:3ddegree1accuracy}
\end{figure}
%%%%%%%%%%%%%%%%%%%%%%%%                                   %%%%%%%%%%%%%%%%%%%%%%%%
%%%%%%%%%%%%%%%%%%%%%%%  Template for single accuracy-dof   %%%%%%%%%%%%%%%%%%%%%%%%
%%%%%%%%%%%%%%%%%%%%%%%%           plot                    %%%%%%%%%%%%%%%%%%%%%%%%
% Usage: either
% plotaccuracy[kappa]
\def \plotaccuracyVdofs[#1]
{
    \begin{tikzpicture}[scale=0.8]
    \begin{loglogaxis}[legend style={at={(0.0,0.0)},anchor=south west},
                       xlabel=Number of DOFs,
                       ylabel=Relative $L^2$ Error,
                       title={$\kappa=$#1},
                       cycle list={
                            {red, solid, mark=*},
                            {blue, densely dotted, mark=square*},
                            {brown, densely dashed, mark=triangle*},
                            {black, dashdotted, mark=diamond*},
                       },
                       filter discard warning=false,
                       scale=0.8]
        \foreach \degree in {1,2,3,4} {
            \addplot+[only marks, thick]
                table [x=ndofs,
                       discard if not={kappa}{#1},
                       y
                       expr={\thisrow{degree}==\degree?\thisrow{nonlocal lu Relative L2 Error}:nan},
                       col sep=comma] {2d_data_accuracy_vs_dofs.csv};
            \addlegendentryexpanded{Degree \degree};
        }
    \end{loglogaxis}
    \end{tikzpicture}
}
%%%%%%%%%%%%%%%%%%%%%%%%                      %%%%%%%%%%%%%%%%%%%%%%%%
%%%%%%%%%%%%%%%%%%%%%%%  Accuracy vs. DOFS     %%%%%%%%%%%%%%%%%%%%%%%%
%%%%%%%%%%%%%%%%%%%%%%%%                      %%%%%%%%%%%%%%%%%%%%%%%%
\begin{figure}[!ht]
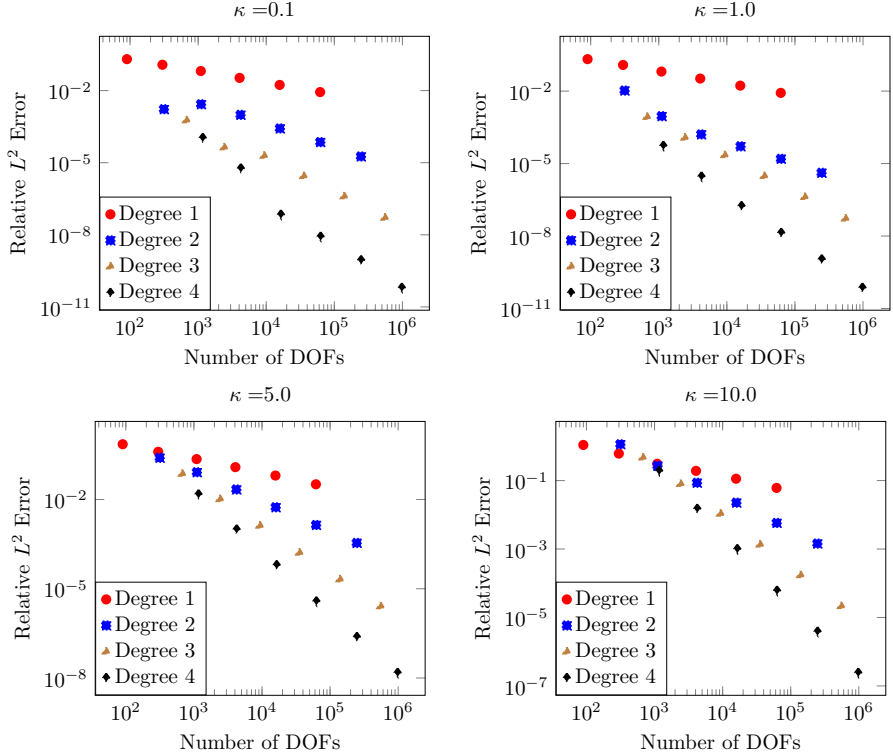

    \begin{center}
        \begin{tabular}{c c}
            \plotaccuracyVdofs[0.1] & \plotaccuracyVdofs[1.0] \\
            \plotaccuracyVdofs[5.0] & \plotaccuracyVdofs[10.0] \\
        \end{tabular}
     \end{center}
 \caption{Relative $L^2$ error of our nonlocal method for degrees 1-4
    as the number of degrees of freedom increase.
    A KSP relative tolerance of $10^{-12}$ was
    used. We used a mesh with no PML sponge region and $\Sigma = [-2, 2]^2$.
}
     \label{fig:accuracyVSdofs}
\end{figure}

Now, we turn to efficient solution of the linear system, focusing on the two-dimensional case.
In table~\ref{tab:pctypeisnone} we see that, even for low wave numbers on coarse meshes with a piecewise linear discretization, solving the system without a preconditioner is not scalable.

\begin{table}[hb]
    \begin{center}
        \begin{filecontents*}{2d_data_pctype_none.csv}
h,kappa,nonlocal none 1e-12 Iteration Number
0.5,0.1,88
0.25,0.1,293
0.125,0.1,675
0.5,1.0,89
0.25,1.0,283
0.125,1.0,519
\end{filecontents*}
\pgfplotstabletypeset[col sep=comma,
                              columns/ndofs/.style={
                                  column name=DOFs,
                                  column type=r,
                              },
                              columns/h/.style={
                                  precision=3,
                                  column name=$h$,
                                  column type=l,
                              },
                              columns/kappa/.style={
                                  fixed,
                                  fixed zerofill,
                                  precision=1,
                                  column name=$\kappa$,
                              },
                              columns/nonlocal none 1e-12 Iteration Number/.style={
                                  column name={Iteration Count},
                                  column type=r,
                              },
                              every head row/.style={before row=\toprule, after row=\midrule},
                              every last row/.style={after row=\bottomrule},
                              ]{2d_data_pctype_none.csv}
    \end{center}
    \caption{Number of iterations to convergence for the nonlocal method
with no preconditioner,
degree 1 polynomials, a KSP relative tolerance of $10^{-12}$,
and GMRES restart set to 200.
Iteration counts were higher with a GMRES restart of 100,
and much higher with the default GMRES restart of 30.}
    \label{tab:pctypeisnone}
\end{table}

We want to demonstrate that the local part of our operator~\eqref{eq:transmission} provides an effective preconditioner, so that our new method can be seen as comparably difficult to solve as the local problem.
As a first approach, we can compute a sparse LU factorization of $A^{L}$ to apply the inverse as a preconditioner for $A$.  The GMRES iteration counts are shown in Figure~\ref{fig:luIterationNumber}.  For a fixed $\kappa$, we see mild \emph{decrease} in the iteration count under mesh refinement.  Moreover, for a fixed mesh, increasing $\kappa$ corresponds only to a slight increase in iteration count.  So, if the underlying transmission operator can be effectively inverted, this will in turn serve as an excellent preconditioner for the system with nonlocal boundary conditions.  
Using a direct method on $A^L$, sparse factorization is typically the dominant cost.  Then, each Krylov iteration requires a sparse matrix-vector product, an FMM evaluation, some quadrature, and solution with the sparse factors.

\begin{figure}[!ht]
% Usage:
%  \plotNonlocalIterationNumberTwoD[degree]
\def \plotNonlocalIterationNumberTwoD[#1]
{
    \begin{tikzpicture}[scale=0.75]
        \begin{semilogxaxis}[
            legend style={at={(1.0,0.0)},anchor=south east},
            xlabel=$h$,
            ylabel=Iteration Number,
            title={degree$=$#1},
            filter discard warning=false,
            cycle list={%
                {red, solid, mark=*},
                {blue, densely dotted, mark=square*},
                {brown, densely dashed, mark=triangle*},
                {black, dashdotted, mark=diamond*},
            }]
        \foreach \k in {0.1,1.0,5.0,10.0} {
            \addplot+[thick]
                table [x=h,
                       y expr={\thisrow{kappa}==\k?\thisrow{nonlocal lu Iteration Number}:nan},
                       col sep=comma] {2d_data_degree#1.csv};
           \addlegendentryexpanded{$\kappa=$\k};
        };
        \end{semilogxaxis}
    \end{tikzpicture}
}
\begin{center}
    \begin{tabular}{c c}
        \plotNonlocalIterationNumberTwoD[1] & \plotNonlocalIterationNumberTwoD[2]  \\
        \plotNonlocalIterationNumberTwoD[3] & \plotNonlocalIterationNumberTwoD[4] \\
    \end{tabular}
\end{center}
\caption{GMRES iteration counts for a two-dimensional mesh using LU factorization for $A^L$ as a preconditioner for various values of $\kappa$.  Fixing $\kappa$ and refining the mesh (right to left) leads to a slight decrease in iteration count, while fixing a mesh and increasing $\kappa$ leads to a mild increase in iteration count.}
\label{fig:luIterationNumber}\end{figure}
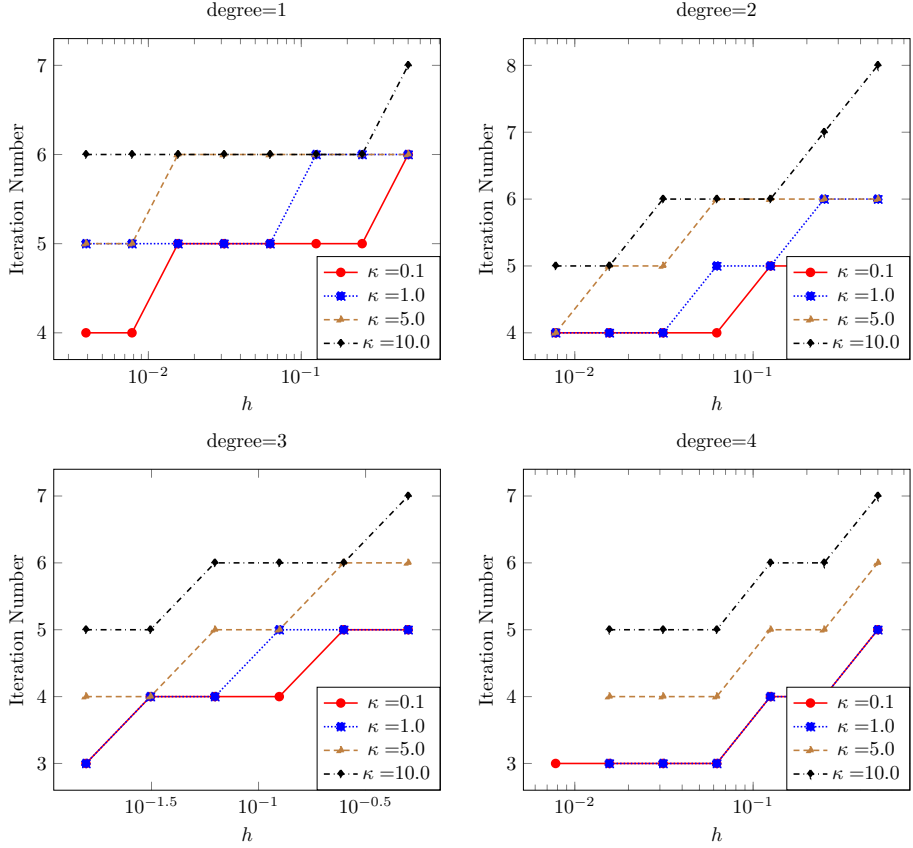

At large enough scale, one might wish to (approximately) invert $A^{L}$ with an iterative method rather than factorization.  To move in this direction, we used \texttt{gamg}~\cite{adams2004algebraic}, a \petsc{}-accessible algebraic multigrid scheme that supports complex arithmetic.  This performed admirably at low wave number ($\kappa \lesssim 1$), but not beyond this.
We were able to tackle higher wave numbers using the approach in~\cite{olson2010smoothed}.  The Laplacian has eigenmodes that become increasingly oscillatory as the eigenvalues increase.  The indefinite Helmholtz operator shifts the eigenvalues (with the same eigenmodes) leftward in the complex plane.  Hence, the eigenvalues closest to zero correspond to certain higher-frequency modes for Helmholtz.  The technique in~\cite{olson2010smoothed} approximates this oscillatory near null space with plane waves.   To apply this method, we wrapped PyAMG~\cite{OlSc2018} as a \software{PETSc4Py} preconditioner.   We applied a fixed number of W-multicycles, augmented with plane~waves in the same way as in~\cite{OlSc2018}, to $A^L$ as a preconditioner for the overall system.  Figure~\ref{fig:EXAMPLEPYAMG} shows the results we obtained.  The preconditioner is very effective at low $\kappa$ but requires more iterations for larger ones.  However, we see that applying more W-cycles within the preconditioner typically leads to a lower outer iteration count.  Comparing Figure~\ref{fig:EXAMPLEPYAMG} to Figure~\ref{fig:luIterationNumber} suggests that the difference in iteration counts follows from the difficulty in obtaining an effective iterative method for the regular Helmholtz operator rather than new difficulties presented by our nonlocal boundary condition.

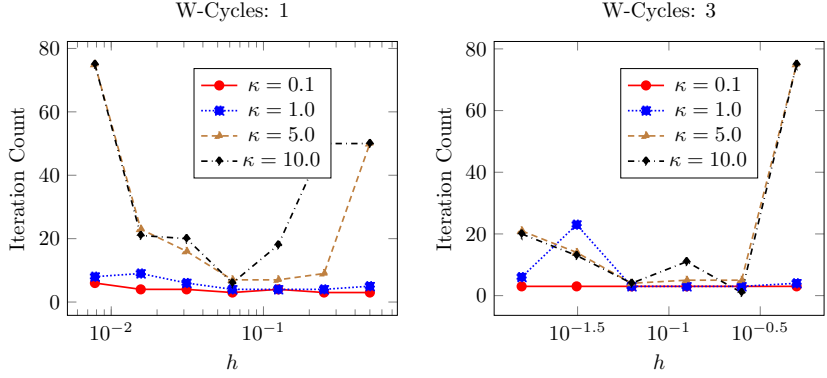
\begin{figure}[!ht]
\begin{center}
    \def \plotwcycles[#1]
    {
        \begin{tikzpicture}[scale=0.8]
        \begin{semilogxaxis}[legend style={at={(0.80,0.50)},anchor=south east},
                             xlabel=$h$,
                             ylabel=Iteration Count,
                             title={W-Cycles: #1},
                             filter discard warning=false,
                             cycle list={%
                                {red, solid, mark=*},
                                {blue, densely dotted, mark=square*},
                                {brown, densely dashed, mark=triangle*},
                                {black, dashdotted, mark=diamond*},
                             },
                             scale=0.8]
            \foreach \kappaval in {0.1,1.0,5.0,10.0} {
                \addplot+[thick]
                    table [x=h,
                           discard if not={pyamgmaxiter}{#1},
                            y expr={\thisrow{kappa}==\kappaval?\thisrow{nonlocal pyamg Iteration Number}:nan},
                           col sep=comma] {pyamg_2d_data.csv};
                \addlegendentryexpanded{$\kappa=\kappaval$};
            };
        \end{semilogxaxis}
        \end{tikzpicture}
    }
    \begin{tabular}{c c}
        \plotwcycles[1] & \plotwcycles[3]
    \end{tabular}
 \end{center}
\caption{Outer GMRES iteration count for various meshes and $\kappa$ values.  We apply PyAMG's smoothed aggregation augmented with plane waves to $A^L$ as a preconditioner for the total system.  At small wave numbers, we see relatively low iteration counts that are fairly flat under mesh refinement.  As $\kappa$ increases, however, far more outer iterations are required to obtain convergence.}
\label{fig:EXAMPLEPYAMG}\end{figure}

Finally, we devise an experiment to demonstrate our method's robustness
with respect to the distance between the scatterer boundary $\Gamma$ and
the truncated boundary $\Sigma$.
We continue to use the circle of radius 1 centered at the origin
as the scatter $\Gamma$.
For side length $s \in \{2.25, 2.5, 3.0, 4.0, 5.0, 6.0\}$,
we truncated the domain as the square $[-\tfrac s2, \tfrac s2]^2$
take out the circle of radius 1.
In Figure~\ref{fig:2dvarying_domain}, we measure
the relative $L^2$ error of our nonlocal method for each domain.  We see no pathologies emerging as the computational domain becomes smaller.

\begin{figure}[H]
\begin{center}
    % Usage:
    % \plotGivenDomainAccuracy[outer side length]
    \def\plotGivenDomainAccuracy[#1]
    {
        \begin{tikzpicture}[scale=0.8]
        \begin{loglogaxis}[legend style={at={(1.0,0.0)},anchor=south east},
                           xlabel=$h$,
                           ylabel=Relative $L^2$ Error,
                           title={$s=#1, \;\operatorname{dist}(\Sigma, \Gamma) = \pgfmathparse{#1/2-1}\pgfmathresult$},
                           cycle list={%
                              {red, solid, mark=*},
                              {blue, densely dotted, mark=square*},
                              {brown, densely dashed, mark=triangle*},
                              {black, dashdotted, mark=diamond*},
                           },
                           filter discard warning=false,
                           scale=0.8]
            \foreach \kappaval in {0.1,1.0,5.0,10.0} {
                \addplot+[thick]
                    table [x=h,
                           discard if not={Outer Side Length}{#1},
                           y expr={\thisrow{kappa}==\kappaval?\thisrow{nonlocal lu Relative L2 Error}:nan},
                           col sep=comma] {2d_data_varying_domain.csv};
                \addlegendentryexpanded{$\kappa=\kappaval$};
            };
        \end{loglogaxis}
        \end{tikzpicture}
    }
    \begin{tabular}{c c}
        \plotGivenDomainAccuracy[2.25] & \plotGivenDomainAccuracy[2.5] \\
        \plotGivenDomainAccuracy[3.0] & \plotGivenDomainAccuracy[4.0] \\
        \plotGivenDomainAccuracy[5.0] & \plotGivenDomainAccuracy[6.0] \\
    \end{tabular}
 \end{center}
\caption{Relative $L^2$ error of the nonlocal method
using degree 1 polynomials
as $\kappa$ and the truncated domain size vary.
A KSP relative tolerance of $10^{-12}$ was used.
}
\label{fig:2dvarying_domain}\end{figure}
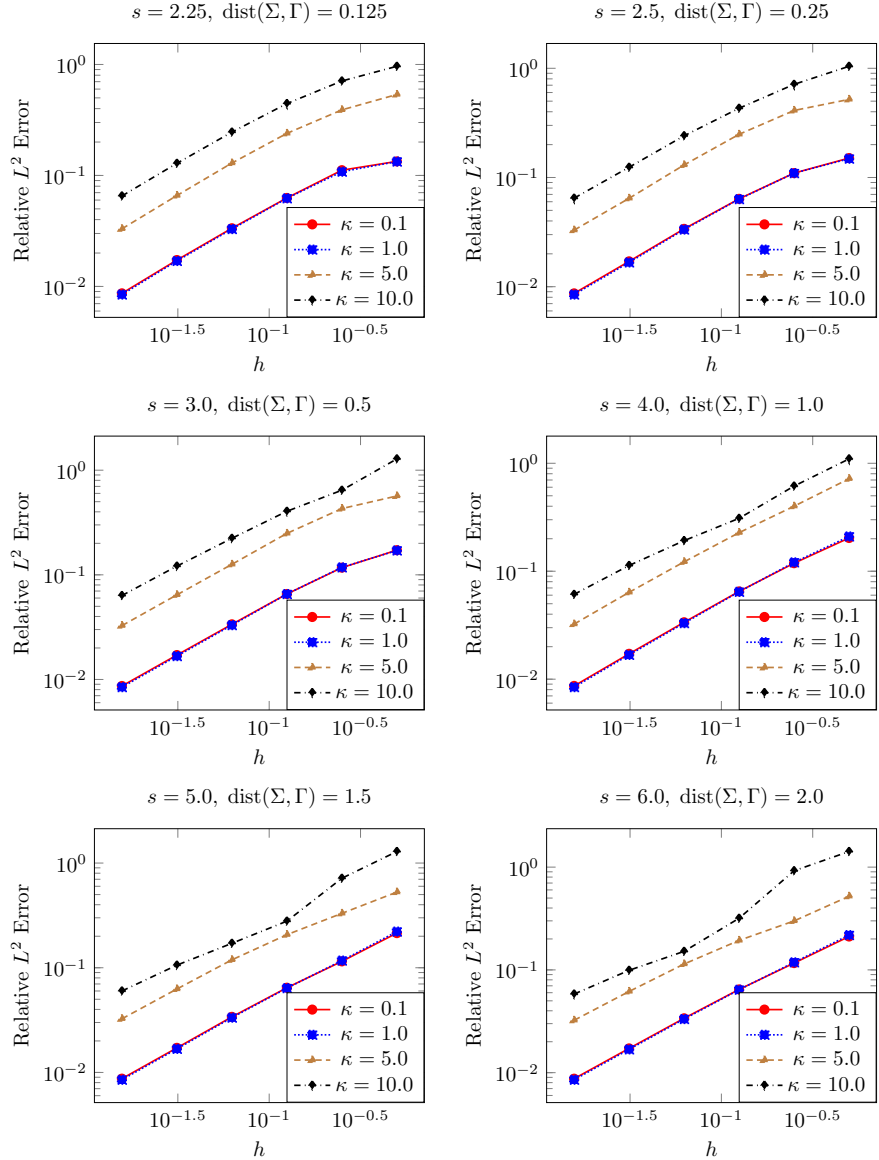

\section{Conclusions and future work}
We have proposed a new nonlocal boundary condition for exterior Helmholtz problems.  This condition, based on Green's formula and expressed in terms of layer potentials, works in general unstructured geometry in two and three dimensions.  Thanks to a G{\aa}rding inequality, we have optimal finite element error estimates under standard conditions.  The nonlocal terms are amenable to approximation by fast multipole expansions, and the discrete system can be readily preconditioned by its local part.  In the future, it should be possible to extend the analysis to handle inexactness in evaluating the boundary terms.  Moreover, we anticipate being able to apply this technique to a much broader class of problems such as exterior curl-curl problems.  Additionally, we are working to integrate layer potentials with Firedrake's top-level language to make it easier to apply the method.

\bibliographystyle{siamplain}
\bibliography{bib}
\end{document}

%% file: shared.tex
\usepackage{booktabs}
\usepackage{lipsum}
\usepackage{amsfonts}
\usepackage{graphicx}
\usepackage{epstopdf}
\usepackage{algorithmic}
\ifpdf
  \DeclareGraphicsExtensions{.eps,.pdf,.png,.jpg}
\else
  \DeclareGraphicsExtensions{.eps}
  \fi

\usepackage{hyphenat}
\usepackage{amssymb}
\usepackage{tikz}
\usetikzlibrary{calc}
\usepackage{pgfplots}
\usepackage{pgfplotstable}
\pgfplotstableset{create on use/Build/.style={create
    col/expr={\thisrow{SNESFunctionEval} + \thisrow{SNESJacobianEval}}}}
\usepackage{listings}
\pgfplotsset{compat=1.14}

\pgfplotsset{
    discard if not/.style 2 args={
        x filter/.code={
            \edef\tempa{\thisrow{#1}}
            \edef\tempb{#2}
            \ifx\tempa\tempb
            \else
                \def\pgfmathresult{inf}
            \fi
        }
    }
}

\newcommand{\software}[1]{\textsc{#1}}
\newcommand{\pytential}{\software{Pytential}}
\newcommand{\firedrake}{Firedrake}
\newcommand{\petsc}{\software{PETSc}}

\newcommand{\bfx}{\mathbf{x}}
\newcommand{\bfb}{\mathbf{b}}

\usepackage{enumitem}
\setlist[enumerate]{leftmargin=.5in}
\setlist[itemize]{leftmargin=.5in}

\newsiamremark{remark}{Remark}
\newsiamremark{hypothesis}{Hypothesis}
\crefname{hypothesis}{Hypothesis}{Hypotheses}
\newsiamthm{claim}{Claim}

\headers{FEM for Helmholtz equations with a nonlocal boundary condition}{R. C. Kirby, A. Kl\"ockner, and B. Sepanski}

\title{Finite elements for Helmholtz equations with a nonlocal boundary condition\thanks{Submitted to the editors DATE.
\funding{This work was funded by NSF awards SHF-1909176 and SHF-1911019.
This material is based upon work
supported by the U.S. Department of Energy, Office of
Science, Office of Advanced Scientific Computing
Research, Department of Energy Computational Science
Graduate Fellowship under Award Number DE-SC0021110.}}}

\author{Robert C. Kirby\thanks{Baylor University, Department of Mathematics, Waco, TX
  (\email{robert\_kirby@baylor.edu}, \url{https://sites.baylor.edu/robert_kirby/}).}
\and Andreas Kl\"ockner\thanks{Department Computer Science, University of Illinois at Urbana-Champaign, Urbana, IL 
  (\email{andreask@illinois.edu})}
\and Ben Sepanski\thanks{University of Texas, Department of Computer
    Science, Austin, TX
    (\email{ben\_sepanski@utexas.edu})}}

\usepackage{amsopn}